\documentclass[a4paper,10pt]{amsart}

\usepackage{amssymb,amsthm,amsmath,verbatim}
\usepackage{t1enc}
\usepackage{enumerate}
\usepackage{color}
\usepackage{tikz}
\usepackage{breqn}
\usepackage{xmpmulti}
\usepackage{psfrag}

\usepackage{todonotes}
 \usepackage[usenames,dvipsnames]{pstricks}
 \usepackage{epsfig}
\usepackage{pst-grad}
\usepackage{pstricks,pst-node}
\usepackage{float}

\usepackage[normalem]{ulem}

\usepackage[left=3.5 cm, right=3.5cm]{geometry}
\numberwithin{equation}{section}

\newcommand{\force}{{\hspace{0.02 cm}\Vdash}}

\newtheorem{prop}{Proposition}[section]

\newtheorem{fact}[prop]{Fact}
\newtheorem{cor}[prop]{Corollary}

\newtheorem{thm}[prop]{Theorem}

\newtheorem{clm}[prop]{Claim}
\newtheorem{obs}[prop]{Observation}
\newtheorem{subclaim}{Subclaim}[prop]
\newtheorem{prob}[prop]{Problem}

\theoremstyle{definition}
\newtheorem{dfn}[prop]{Definition}

\newcommand{\mc}[1]{\mathcal{#1}}
\newcommand{\mb}[1]{\mathbb{#1}}

\newcommand{\oo}{\omega}
\newcommand{\uhr}{\upharpoonright}
\newcommand{\omg}{{\omega_1}}

\def\<{\left\langle}
\def\>{\right\rangle}
\def\br#1;#2;{\bigl[ {#1} \bigr]^ {#2} }

\newcommand{\mf}[1]{\mathfrak{#1}}

\newcommand{\setm}{\setminus}

\newcommand{\subs}{\subset}
\newcommand{\dom}{\operatorname{dom}}

\newcommand{\vareps}{\varepsilon}

\newcommand{\bb}{\mathbb}

\title{Extremal triangle-free and odd-cycle-free colourings of uncountable graphs}

\date{\today}

\author{Chris Lambie-Hanson}
  \address{Department of Mathematics and Applied Mathematics, Virginia Commonwealth University, Harris Hall, Office 4173}
 \email[Corresponding author]{cblambiehanso@vcu.edu}

 \urladdr{https://people.vcu.edu/$\sim$cblambiehanso/}

  \author{D\'aniel T. Soukup}
 \email{daniel.t.soukup@gmail.com}

 \urladdr{https://danieltsoukup.github.io}

 \subjclass[2010]{Primary 03E02. Secondary 05C63, 03E05.}
 \keywords{Ramsey theory, regressive colourings, triangle-free colourings,
  uncountable graphs}

\makeatletter
\newtheorem*{rep@theorem}{\rep@title}
\newcommand{\newreptheorem}[2]{%
\newenvironment{rep#1}[1]{%
 \def\rep@title{#2 \ref{##1}}%
 \begin{rep@theorem}}%
 {\end{rep@theorem}}}
\makeatother

\newreptheorem{corollary}{Corollary}
\newreptheorem{theorem}{Theorem}

\begin{document}
\begin{abstract}
  The optimality of the Erd\H{o}s-Rado theorem for pairs is witnessed by
  the colouring $\Delta_\kappa : [2^\kappa]^2 \rightarrow \kappa$ recording the
  least point of disagreement between two functions. This colouring has
  no monochromatic triangles or, more generally, odd cycles. We investigate
  a number of questions investigating the extent to which $\Delta_\kappa$ is an
  \emph{extremal} such triangle-free or odd-cycle-free colouring. We begin
  by introducing the notion of $\Delta$-regressive and almost $\Delta$-regressive
  colourings and studying the structures that must appear as monochromatic
  subgraphs for such colourings. We also consider the question as to whether
  $\Delta_\kappa$ has the minimal cardinality of any \emph{maximal}
  triangle-free or odd-cycle-free colouring into $\kappa$. We resolve the
  question positively for odd-cycle-free colourings.
\end{abstract}
\maketitle

\section{Introduction}

The starting point of our paper is the classical Erd{\H o}s-Rado partition
relation:
\[
  (2^\kappa)^+\to (\kappa^+)^2_\kappa,
\]
i.e., for any colouring $c:[(2^\kappa)^+]^2\to \kappa$, there is a
$c$-monochromatic subset of $(2^\kappa)^+$ of size $\kappa^+$
\cite{erdos_rado_partition_calculus}. This result
is optimal, in the sense that
\[
  2^\kappa\not\to (3)^2_\kappa.
\]
This negative relation is witnessed by the following natural colouring. For two
functions $x,y$ whose domains are sets of ordinals, let
\[
  \Delta(x,y)=\min\{i \in \dom x\cap \dom y:x(i)\neq y(i)\}
\]
if such an $i$ exists (otherwise $\Delta(x,y)$ is undefined).
We let $\Delta_\kappa$ denote the restriction of $\Delta$ to $[2^\kappa]^2$,
where $2^\kappa$ denotes the set of all functions from $\kappa$ to $2$.
The following is immediate.

\begin{fact}
For any infinite $\kappa$, the colouring $\Delta_\kappa:[2^\kappa]^2\to \kappa$ has no monochromatic triples or odd cycles.
\end{fact}

We are interested in the extent to which $\Delta$ is a \emph{minimal}
triangle-free or odd-cycle-free colouring on $[2^\kappa]^2$. One way
in which this question can be made precise is via the following
definition.

\begin{dfn}
We say that $c:[2^\kappa]^2\to \kappa$  is \emph{$\Delta$-regressive} if  $$c(x,y)<\Delta(x,y)$$ for all $x, y$ such that $\Delta(x,y)>0$.
\end{dfn}

So our primary question about such colourings is: What monochromatic subgraphs must
appear in a $\Delta$-regressive colouring on $[2^\kappa]^2$? We will see in
Section \ref{delta-regressive_section} that, if $\kappa > \omega$, then such
a colouring must have monochromatic cycles of every length. This will lead us
to introduce the notion of \emph{almost $\Delta$-regressive colourings} and study
analogous questions with respect to this larger class of colourings.

In particular, we will prove that
\begin{itemize}
  \item the existence of an almost $\Delta$-regressive colouring on
  $[2^\kappa]^2$ with no monochromatic odd cycles is equivalent to the
  existence of $\mu < \kappa$ such that $2^\mu = 2^\kappa$;
  \item it is consistent relative to the consistency of a measurable cardinal
  that, for instance, every almost $\Delta$-regressive colouring on
  $[2^{\omega_1}]^2$ has a monochromatic set of size $\aleph_1$;
  \item every almost $\Delta$-regressive $\kappa$-Borel colouring on
  $[2^\kappa]^2$ has a monochromatic set of size $\kappa$.
\end{itemize}

\medskip

In Section \ref{maximal_colouring_section}, we look at maximality properties
of the colouring $\Delta_\kappa$ on $[2^\kappa]^2$.

\begin{dfn}
  Suppose that $\mc H$ is a collection of graphs and $c:[X]^2 \rightarrow
  \kappa$ is a colouring.
  \begin{enumerate}
    \item We say that $c$ is \emph{$\mathcal{H}$-free} if, for every $G \in
    \mathcal{H}$, there is no subgraph $E \subseteq [X]^2$ isomorphic to $G$
    such that $c \restriction E$ is constant.
    \item We say that $c$ is \emph{maximal $\mc H$-free (into $\kappa$)} if
    it is $\mc H$-free but, for any $y \notin X$ and any $d:[X \cup \{y\}]^2
    \rightarrow \kappa$ extending $c$, $d$ is not $\mc H$-free.
  \end{enumerate}
\end{dfn}

We are mostly concerned with $\mc H$ being the single 3-cycle or the collection of
all odd cycles. We will see that $\Delta_\kappa$ is a maximal 3-cycle-free
colouring (and hence a maximal odd-cycle-free colouring) into $\kappa$. Our
primary question here is whether there exist \emph{smaller} such maximal
colourings, i.e., whether there are maximal 3-cycle-free or odd-cycle-free
colourings $c:[X]^2 \rightarrow \kappa$ in which $|X| < 2^\kappa$.
We answer this question negatively for odd-cycle-free colourings, in the process
providing a characterization of odd-cycle-free colourings. The question
for 3-cycle-free colourings remains open.

In Section \ref{further_remarks_section}, we prove a couple of results about
3-cycle-free colourings $c:[\omega_1]^2 \rightarrow \omega$. In particular,
we show that, if $\diamondsuit$ holds, then there is a colouring
$c:[\omega_1]^2 \rightarrow \omega$ so that each colour class
$G_n = c^{-1}(\{n\})$ is a Hajnal-M\'{a}t\'{e} graph with uncountable
chromatic number, extending previous results of Komj\'{a}th (\cite{komjath_note_on_hajnal_mate}, \cite{komjath_second_note}).

We conclude in Section \ref{problem_section} by collecting a number of
problems that remain open.

Our notation and terminology is for the most part standard. If
$X$ is a set and $\theta$ is a cardinal, then
$[X]^\theta := \{Y \subseteq X \mid |Y| = \theta\}$. If $c$ is
a function on $[X]^2$, then we will frequently abuse notation and
write $c(x,y)$ in place of $c(\{x,y\})$. If $c:[X]^2
\rightarrow \kappa$ is a colouring and $Y \subseteq X$, then $Y$ is
\emph{$c$-monochromatic} if $c \restriction [Y]^2$ is constant.
A $c$-monochromatic triple (or triangle) is a $c$-monochromatic subset with
exactly three elements. A $c$-monochromatic \emph{cycle} is a finite
injective sequence $\langle x_i \mid i < k \rangle$ of elements of
$X$ (with $k \geq 3$) such that $c$ is constant on the set $\{\{x_i, x_{i+1}\} \mid i + 1 < k\}
\cup \left\{\{x_{k-1}, x_0\}\right\}$. The number $k$ is the
\emph{length} of the cycle. Such a cycle is a $c$-monochromatic
\emph{odd} cycle if $k$ is odd.

\section{Monochromatic subsets in $\Delta$-regressive colourings}
\label{delta-regressive_section}

We begin with the simple observation that there are $\Delta$-regressive
colourings on $[2^\omega]^2$ without monochromatic triples. The following
is easily verified.

\begin{obs}
Define a colouring $c:[2^\omega]^2 \rightarrow \omega$ by letting, for all
$\{f,g\} \in [2^\omega]^2$,
\begin{align*}
c(f,g) =
\begin{cases}
    \Delta(f,g) - 1 & \text{if } \Delta(f,g) > 0 \\
    f(1) + g(1) ~ (\mathrm{mod} ~ 2) & \text{if } \Delta(f,g) = 0
\end{cases}
\end{align*}
Then $c$ is $\Delta$-regressive and has no monochromatic triples.
\end{obs}

However, for $2^\alpha$ with $\alpha > \omega$, the situation is rather different.

\begin{prop} For any $\Delta$-regressive $c:[2^{\oo+1}]^2\to \oo+1$, there is an $m<\oo$ such that, for all lengths $k \geq 3$, $c$  has monochromatic cycles of colour $m$ and length $k$.
\end{prop}
\begin{proof}
  Suppose for sake of contradiction that, for all $m < \omega$, there is a
  natural number $k_m \geq 3$ such that $c$ has no monochromatic cycles of
  colour $m$ and length $k_m$.

  \begin{clm}
    For all $m \geq 1$ and $\{f,g\} \in [2^{\omega + 1}]^2$, if
    $\Delta(f,g) \geq m$, then $c(f,g) \geq m-1$.
  \end{clm}

  \begin{proof}
    The proof is by induction on $m$. The claim is trivially true for
    $m = 1$. Fix $m \geq 1$ and suppose
    we have established the claim for $m$. In particular, for all
    $\{h, h'\} \in [2^{\omega + 1}]^2$, if $\Delta(h,h') = m$, then,
    by the fact that $c$ is $\Delta$-regressive, we have
    $c(h,h') = m - 1$. To establish the claim for $m+1$, assume for sake of
    contradiction that we can find $f,g \in 2^{\omega + 1}$ such that
    $\Delta(f,g) \geq m + 1$ but $c(f,g) < m$. By our inductive hypothesis,
    it follows that $c(f,g) = m - 1$.

    We will reach a contradiction by finding a monochromatic cycle of
    colour $m-1$ and length $k_{m-1}$. If $k_{m-1}$ is even, then simply
    choose an injective sequence $\langle h_j \mid j < k_{m-1} \rangle$
    such that $\Delta(h_j, h_{j+1}) = m$ for all $j < k_{m-1} - 1$
    (and hence also $\Delta(h_{k_{m-1} - 1}, h_0) = m$). Then
    $\langle h_j \mid j< k_{m-1} \rangle$ is the desired monochromatic cycle.

    If $k_{m-1}$ is odd, then choose an injective sequence
    $\langle h_j \mid j < k_{m-1} - 2 \rangle$ such that
    \begin{itemize}
      \item $\Delta(h_j, h_{j+1}) = m$ for all $j < k_{m-1} - 3$;
      \item $\Delta(f,h_j) = m$ for all even $j < k_{m-1} - 2$
      (and hence $\Delta(g, h_j) = m$ for all even $j < k_{m-1} - 2$ as well).
    \end{itemize}
    Then $\langle f,g \rangle ^\frown \langle h_j \mid j < k_{m-1} - 2 \rangle$
    is the desired monochromatic cycle.
  \end{proof}

  Now fix $f,g \in 2^{\omega + 1}$ with $\Delta(f,g) = \omega$. By the claim,
  $c(f,g) \geq m-1$ for all $m < \omega$, and hence $c(f,g) = \omega$,
  contradicting the fact that $c$ is $\Delta$-regressive.
\end{proof}

For this reason, the notion of $\Delta$-regressive seems to be too strong to
be of interest for certain questions. We therefore introduce a natural weakening.

\begin{dfn}
  We say that $c:[2^\kappa]^2\to \kappa$ is  \emph{almost $\Delta$-regressive}
  if there is $\mu < \kappa$ such that
  \[
    c(x,y)<\max\{\Delta(x,y),\mu\}
  \]
  for all $x\neq y$.
\end{dfn}

Once again, we ask what monochromatic subgraphs must appear in almost
$\Delta$-regressive colourings, and in particular
whether they must contain monochromatic triangles or odd cycles. Let us first make
an easy observation indicating that consistently there are almost
$\Delta$-regressive colourings avoiding all monochromatic odd cycles.

\begin{prop} \label{not_weak_ch_prop}
  Suppose that $\mu < \kappa$ are infinite cardinals and $2^\mu = 2^\kappa$.
  Then there is an almost $\Delta$-regressive colouring $c:[2^\kappa]^2 \rightarrow
  \kappa$ with no monochromatic odd cycles.
\end{prop}

\begin{proof}
  We will in fact find such a function $c$ mapping into $\mu$; it will therefore
  trivially be almost $\Delta$-regressive, as witnessed by $\mu$.

  Fix a bijection $F:2^\kappa \rightarrow 2^\mu$, and define $c:[2^\kappa]^2
  \rightarrow \mu$ by letting $c(x,y) = \Delta(F(x), F(y))$. It is immediate
  that $c$ is as desired.
\end{proof}

In light of this fact, the most interesting case seems to be when $\kappa$ is
uncountable and $2^\mu < 2^\kappa$ for all $\mu < \kappa$.
Note that, under these assumptions, for any function $c:[2^\kappa]^2
\rightarrow \mu$ with $\mu < \kappa$, the Erd\H{o}s-Rado theorem yields
monochromatic subsets of cardinality $\mu^+$. Therefore, unlike the situation in
Proposition~\ref{not_weak_ch_prop}, an example of an almost
$\Delta$-regressive function avoiding small monochromatic sets,
if there is one, needs to have range of size $\kappa$.

\begin{thm} \label{almost_regressive_odd_cycle_thm}
  Suppose that $\kappa$ is an uncountable cardinal and $2^\mu < 2^\kappa$ for all
  $\mu < \kappa$. Suppose also that $c:[2^\kappa]^2 \rightarrow \kappa$ is \emph{almost $\Delta$-regressive}. Then there are $c$-monochromatic odd cycles.
\end{thm}

\begin{proof}
  Suppose this is not the case, and let $\mu < \kappa$ be such that
  $c(x,y) < \max\{\Delta(x,y), \mu\}$ for all $x \neq y$. For each colour
  $\xi < \kappa$, the colour class $c^{-1}(\{\xi\})$ is odd-cycle-free
  and hence, as a graph, bipartite. We can therefore partition $2^\kappa$ into disjoint
  sets $2^\kappa = X^0_\xi \dot{\cup} X^1_\xi$ so that no pair
  from $X^i_\xi$ has colour $\xi$. Define $\pi:2^\kappa\to 2^\kappa$ by letting,
  for all $x \in 2^\kappa$ and all $\xi < \kappa$,
  $\pi(x)(\xi)=i$ if $x\in X^i_\xi$. Thus, if $\{x,y\} \in
  [2^\kappa]^2$ and $c(x,y) = \xi$, then $\pi(x)(\xi) \neq \pi(y)(\xi)$. Therefore $\pi$ must be injective, and, if $\Delta(x,y) \geq \mu$, then
  \[
    \Delta(\pi(x),\pi(y))\leq c(x,y)<\Delta(x,y).
  \]

 So, for any $x\neq y$, there is a minimal $n=n_{xy}<\oo$ so that
 \[
  \xi_{xy} := \Delta(\pi^n(x),\pi^n(y))<\mu.
 \]
 Indeed, otherwise we could find a strictly decreasing infinite sequence of
 ordinals.

 Apply the Erd\H os-Rado theorem to the map $g:[2^\kappa]^2 \rightarrow
 \omega \times \mu$ taking a pair $\{x,y\}$ to $(n_{xy},\xi_{xy})$ to find
 three distinct $x,y,z$ in $2^\kappa$ and $(n,\xi)$ in $\omega \times \mu$ so
 that $g `` [\{x,y,z\}]^2= \{(n,\xi)\}$. This means that the function
 $\Delta$ is constant, with value $\xi$, on the 3-element set
 $\{\pi^n(x),\pi^n(y),\pi^n(z)\}$, a contradiction to the fact that $\Delta$
 is triangle-free.
\end{proof}

At this point, we do not know if Theorem \ref{almost_regressive_odd_cycle_thm}
can be improved to yield the necessary existence of monochromatic triples.
The above argument can easily be generalized to prove the following result,
whose verification we leave to the reader.

\begin{thm}
  Suppose that $2 \leq \nu < \kappa$ are cardinals, with $\kappa$ uncountable,
  $2^\mu < 2^\kappa$ for all
  $\mu < \kappa$, and $c:[\nu^\kappa]^2 \rightarrow \kappa$ is
  almost $\Delta$-regressive. Then there is $\xi < \kappa$ such that the graph
  $(\nu^\kappa, c^{-1}(\{\xi\}))$ has chromatic number greater than $\nu$.
\end{thm}

If we assume that $\kappa$ is a large cardinal, we can obtain a stronger
result. Recall the following large cardinal notion.

\begin{dfn}
  Let $\kappa$ be an uncountable cardinal.
  \begin{enumerate}
    \item A function $f:[\kappa]^2 \rightarrow \kappa$ is \emph{regressive}
    if $f(\alpha, \beta) < \alpha$ for all $0 < \alpha < \beta < \kappa$.
    \item $\kappa$ is an \emph{almost ineffable cardinal} if, for every
    regressive function $f:[\kappa]^2 \rightarrow \kappa$, there is
    an $f$-monochromatic set of cardinality $\kappa$.
  \end{enumerate}
\end{dfn}

We note that this is different from the usual definition of almost ineffability
but was proven to be equivalent by Baumgartner
\cite[Theorem 5.2]{baumgartner_ineffability}. To put almost ineffability into
the context of possibly more familiar large cardinal notions, it is easily
seen that all measurable cardinals are almost ineffable and all almost
ineffable cardinals are weakly compact.

\begin{prop}
  Suppose that $\kappa$ is an almost ineffable cardinal and
  $c:[2^\kappa]^2\to \kappa$ is almost $\Delta$-regressive.  Then there are
  $c$-monochromatic subsets of size $\kappa$.
\end{prop}

Note that, in general, we cannot hope to find monochromatic sets of size $\kappa^+$
by the existence of Sierpinski colourings from $[2^\kappa]^2$ to $2$ witnessing
$2^\kappa \not\rightarrow (\kappa^+)^2_2$.

\begin{proof}
  Let $\mu < \kappa$ witness that $c$ is almost $\Delta$-regressive. Fix
  an injective sequence $\langle y_\xi \mid \xi < \kappa \rangle$ from
  $2^\kappa$ so that $\Delta(y_\zeta,y_\xi)=\zeta$ for all $\zeta<\xi<\kappa$.
  Notice that, for all $\mu \leq \zeta < \xi < \kappa$, we have
  $c(y_\zeta, y_\xi) < \zeta$. Now define a function $f:[\kappa]^2 \rightarrow
  \kappa$ by letting $f(\zeta, \xi) = c(y_\zeta, y_\xi)$ if
  $\mu \leq \zeta < \xi < \kappa$ and $f(\zeta, \xi) = 0$ otherwise.
  Then $f$ is regressive, so, by the almost ineffability of $\kappa$,
  there is an $f$-monochromatic set $A \subseteq \kappa$ of size
  $\kappa$. But then $\{y_\xi \mid \xi \in A \setminus \mu\}$ is a $c$-monochromatic
  set of size $\kappa$.
\end{proof}

We can use similar ideas to obtain consistency results for small cardinals
$\kappa$, assuming the consistency of a measurable cardinal. For example, a special
case of the following theorem yields the
consistency of the assertion that every almost $\Delta$-regressive colouring
of $[2^{\omega_1}]^2$ has uncountable monochromatic subsets.

\begin{thm}
  Suppose that $\kappa < \lambda$ are regular uncountable cardinals, with
  $\lambda$ measurable, and let $\bb{P} = \mathrm{Add}(\kappa, \lambda)$
  be the poset to add $\lambda$-many Cohen subsets to $\kappa$.
  Then, in $V^{\bb{P}}$, every almost $\Delta$-regressive colouring
  $c:[2^\kappa]^2 \rightarrow \kappa$ has monochromatic subsets of size $\kappa$.
\end{thm}

\begin{proof}
  We think of conditions in $\bb{P}$ as being partial functions from
  $\lambda \times \kappa$ to $2$ of cardinality less than $\kappa$, ordered
  by reverse inclusion.
  For $p \in \bb{P}$, let
  \[
    D_p = \{\alpha < \lambda \mid \text{there is } i < \kappa
    \text{ such that } (\alpha, i) \in \dom(p)\}.
  \]
  Slightly abusing notation, if $D \subseteq D_p$, then let $p \restriction
  D$ denote $p \restriction (\dom(p) \cap (D \times \kappa))$.
  For $\alpha < \lambda$, let $\dot{f}_\alpha$ be the canonical $\bb{P}$-name
  for the $\alpha^{\mathrm{th}}$ Cohen subset added. In other words, for
  all $p \in \bb{P}$ and $i < \kappa$, if $(\alpha, i) \in \dom(p)$,
  then $p \Vdash ``\dot{f}_\alpha(i) = p(\alpha, i)"$.
  Fix a $\bb{P}$-name $\dot{c}$ for an almost $\Delta$-regressive colouring from
  $[2^\kappa]^2$ to $\kappa$, and fix a condition $p_0 \in \bb{P}$ and a
  cardinal $\mu < \kappa$ such that
  \[
    p_0 \Vdash_{\bb{P}}``\dot{c}(x, y) < \max(\Delta(x, y),
    \check{\mu}) \text{ for all } x, y \in 2^\kappa".
  \]
  Fix also a normal measure $U$ over $\lambda$.

  Let $G$ be $\bb{P}$-generic over $V$ with $p_0 \in G$, and work in $V[G]$.
  We will recursively construct
  \begin{itemize}
    \item an increasing sequence $\langle \alpha_\xi \mid \xi < \kappa \rangle$
    of ordinals below $\lambda$;
    \item sets $X_\xi \in U$ and ordinals $i_\xi < \delta_\xi < \varepsilon_\xi
    < \kappa$ for each $\xi < \kappa$, with $\langle \delta_\xi \mid \xi < \kappa
    \rangle$ increasing and continuous;
    \item conditions $q_\xi$ and $p_{\alpha_\xi \beta}$ in $\bb{P}$ for all
    $\xi < \kappa$ and $\beta \in X_\xi$;
    \item a $\subseteq$-increasing sequence of functions $\langle
    g_\xi : \varepsilon_\xi \rightarrow 2 \mid \xi < \kappa \rangle$
  \end{itemize}
  such that for all $\xi < \kappa$ and $\beta \in X_\xi$,
  \begin{enumerate}
    \item $\dom(p_{\alpha_\xi \beta}) \subseteq D_{p_{\alpha_\xi \beta}}
    \times \varepsilon_\xi$;
    \item $\varepsilon_\xi \leq \delta_{\xi + 1}$;
    \item $p_{\alpha_\xi \beta} \Vdash ``\check{\delta}_\xi =
    \Delta(\dot{f}_{\alpha_\xi}, \dot{f}_\beta) \text{ and }
    \dot{f}_\beta \restriction \check{\varepsilon}_\xi = \check{g}_\xi"$;
    \item $p_{\alpha_\xi \beta} \leq p_0$ and $p_{\alpha_\xi \beta} \Vdash
    ``\dot{c}(\dot{f}_{\alpha_\xi},
    \dot{f}_\beta) = \check{i}_\xi"$;
    \item $\{D_{p_{\alpha_\xi \beta}} \mid \beta \in X_\xi\}$ forms a $\Delta$-system
    with root $D_\xi$, and $p_{\alpha_\xi \beta} \restriction D_\xi = q_\xi$;
    \item for all $\zeta < \xi$, we have
    \begin{enumerate}
      \item $X_\xi \cup \{\alpha_\xi\} \subseteq X_\zeta$;
      \item $p_{\alpha_\xi \beta}$ extends both $p_{\alpha_\zeta \alpha_\xi}$
      and $p_{\alpha_\zeta \beta}$;
      \item $p_{\alpha_\zeta \alpha_\xi}, q_\xi \in G$.
    \end{enumerate}
  \end{enumerate}

We set $X_{-1} = \lambda$ and $\delta_{-1} = \mu$, and fix $\varepsilon_{-1}
< \kappa$ such that $\dom(p_0) \subseteq D_{p_0} \times
\varepsilon_{-1}$, and we describe
the general step of the recursion. Suppose that $\xi < \kappa$ is fixed
and that we have constructed the above objects for all $\zeta < \xi$.
Notice that, since $\bb{P}$ is $\kappa$-closed, the construction thus
far all lives in $V$. Let
\[
  r_\xi = \bigcup_{\eta < \zeta < \xi} p_{\alpha_\eta \alpha_\zeta}
  \cup \bigcup_{\zeta < \xi} q_\eta.
\]
(If $\xi = 0$, let $r_\xi = p_0$.) By the closure of $\bb{P}$, $r_\xi$ is
in fact a condition in $\bb{P}$ and is in $G$. Set $\delta_\xi =
\sup\{\varepsilon_\zeta \mid \zeta < \xi\}$,
and note that $\dom(r_\xi) \subseteq D_{r_\xi} \times \delta_\xi$. Move now to $V$.

\begin{clm}
  Let $E_\xi$ be the set of $q \leq r_\xi$ for which there exist
  $\alpha^* < \lambda$, $i^* < \varepsilon^* < \kappa$,
  $g^* : \varepsilon^* \rightarrow 2$,
  $X^* \in U$, and conditions $p_{\alpha^* \beta}$ for $\beta \in X^*$ such
  that
  \begin{itemize}
    \item $i^* < \delta_\xi < \varepsilon^*$;
    \item $\dom(p_{\alpha^*, \beta}) \subseteq D_{p_{\alpha^*, \beta}} \times
    \varepsilon^*$ for all $\beta \in X^*$;
    \item $p_{\alpha^* \beta} \Vdash ``\check{\delta}_\xi =
    \Delta(\dot{f}_{\alpha^*}, \dot{f}_\beta) \text{ and }
    \dot{f}_\beta \restriction \check{\varepsilon}^* = \check{g}^*"$ for
    all $\beta \in X^*$;
    \item $p_{\alpha^* \beta} \Vdash ``\dot{c}(\dot{f}_{\alpha^*},
    \dot{f}_\beta) = \check{i}^*"$ for all $\beta \in X^*$;
    \item $\{D_{p_{\alpha^* \beta}} \mid \beta \in X^*\}$ forms a $\Delta$-system
    with root $D^*$, and $p_{\alpha^* \beta} \restriction D^* = q$ for all
    $\beta \in X^*$;
    \item $X^* \cup \{\alpha^*\} \subseteq \bigcap\{X_\zeta \mid \zeta < \xi\}$;
    \item $q$ extends $p_{\alpha_\zeta \alpha_\xi}$ for all $\zeta < \xi$;
    \item $p_{\alpha^* \beta}$ extends $p_{\alpha_\zeta \beta}$ for all
    $\zeta < \xi$ and $\beta \in X^*$.
  \end{itemize}
  Then $E_\xi$ is dense in $\bb{P}$ below $r_\xi$.
\end{clm}

\begin{proof}
  Let $t \leq r_\xi$ be arbitrary, and let $Y$ be the set of all
  $\beta \in \bigcap\{X_\zeta \mid \zeta < \xi\}$ such that
  $\beta > \max\{\alpha_\zeta, \sup(D_t)\}$ and
  $D_{p_{\alpha_\zeta \beta}} \setminus D_{q_\zeta}$ is disjoint
  from $D_t$ for all $\zeta < \xi$. Note that, by our recursion hypothesis,
  we have that, for all $\zeta < \xi$, the sequence
  $\langle D_{p_{\alpha_\zeta \beta}} \setminus D_{q_\zeta} \mid
  \beta \in X_\zeta \rangle$ consists of pairwise disjoint sets, and
  therefore we know that $Y \in U$.
  Let $\alpha^* = \min(Y)$, and let $t^* = t \cup \bigcup_{\zeta < \xi}
  p_{\alpha_\zeta \alpha^*}$. By our choice of $\alpha^*$,
  $t^*$ is a function and thus a condition in $\bb{P}$.

  Let $Y^* = Y \setminus (\alpha + 1)$. For all $\beta \in Y^*$,
  let $t^*_\beta = t^* \cup \bigcup\{p_{\alpha_\zeta \beta} \mid \zeta < \xi\}$.
  By our choice of $Y^*$, $t^*_\beta$ is a function and hence a condition in
  $\bb{P}$. Let $g^- = \bigcup_{\zeta < \xi} g_\zeta$, and notice that
  $t^*_\beta \Vdash ``\dot{f}_{\alpha^*} \restriction \check{\delta}_\xi =
  \dot{f}_\beta \restriction \check{\delta}_\xi = \check{g}^-"$, and $t^*_\beta$
  does not decide the value of $\dot{f}_{\alpha^*}(\eta)$ or
  $\dot{f}_\beta(\eta)$ for any $\delta_\xi \leq \eta < \kappa$.
  We can therefore fix a condition $p_{\alpha^* \beta} \leq t^*_\beta$ such that
  \begin{itemize}
    \item $p_{\alpha^* \beta} \Vdash ``\Delta(\dot{f}_{\alpha^*} \dot{f}_\beta)
    = \check{\delta}_\xi"$; and
    \item $p_{\alpha^* \beta}$ decides the value of $\dot{c}(\dot{f}_{\alpha^*},
    \dot{f}_\beta)$ to be equal to some $i^*_{\beta} < \delta_\xi$.
  \end{itemize}
  Let $\epsilon_{\alpha^* \beta} < \kappa$ be such that $\dom(p_{\alpha^* \beta})
  \subseteq D_{p_{\alpha^* \beta}} \times \epsilon_{\alpha^*, \beta}$. Without
  loss of generality, we may assume that $\{\beta\} \times
  \epsilon_{\alpha^* \beta} \subseteq
  \dom(p_{\alpha^* \beta})$, so we can define a function
  $g_{\alpha^* \beta} : \epsilon_{\alpha^* \beta} \rightarrow 2$ by letting
  $g_{\alpha^* \beta}(j) = p_{\alpha^* \beta}(\beta, j)$ for all $j <
  \epsilon_{\alpha^* \beta}$.
  Now consider the map $h$ that sends each $\beta \in Y^*$ to the tuple
  $\left \langle p_{\alpha^* \beta} \restriction (D_{p_{\alpha^* \beta}}
  \cap \beta), ~ i^*_\beta, ~ \epsilon_{\alpha^* \beta}, ~ g_{\alpha^* \beta}
  \right \rangle$. Then $h$ can be coded as a regressive function, defined
  on a set in $U$, so, by the normality of $U$, we can find a
  set $X^* \subseteq Y^*$, a condition
  $q \in \bb{P}$, ordinals $i^*$ and $\varepsilon^*$ such that
  $i^* < \delta_\xi < \varepsilon^* < \kappa$, and a function
  $g^* : \varepsilon^* \rightarrow 2$ such that
  \begin{itemize}
    \item $X^* \in U$;
    \item $h(\beta) = \langle q, ~ i^*, ~ \varepsilon^*, ~ g^* \rangle$
    for all $\beta \in X^*$;
    \item $D_{p_{\alpha^* \beta}} \subseteq \beta'$ for all $\beta < \beta' \in
    X^*$.
  \end{itemize}
  Then $q \leq t$ is as in the statement of the claim, as witnessed by $\alpha^*,
  i^*, \varepsilon^*, g^*, X^*$, and $\{p_{\alpha^* \beta} \mid \beta \in X^*\}$.
\end{proof}

Now move back to $V[G]$. By the claim and the fact that $r_\xi \in G$, we can
find $q_\xi \in E_\xi \cap G$, as witnessed by $\alpha_\xi < \lambda$,
$i_\xi < \varepsilon_\xi < \kappa$, $g_\xi : \varepsilon_\xi \rightarrow 2$,
$X_\xi \in U$, and conditions
$p_{\alpha_\xi \beta} \in \bb{P}$ for all $\beta \in X_\xi$. It is easily
verified that these objects are as desired, thus completing the recursive
construction.

Now the map sending $\delta_\xi$ to $i_\xi$ for all $\xi < \kappa$ is a
regressive function defined on a club of ordinals in $\kappa$, so there is a
fixed $i < \kappa$ and a stationary $S \subseteq \kappa$ such that
$i_\xi = i$ for all $\xi \in S$. It follows that, in $V[G]$,
$\{f_{\alpha_\xi} \mid \xi \in S\}$ is a monochromatic subset for $c$ of size
$\kappa$.
\end{proof}

We end this section with a discussion indicating that \emph{sufficiently
nice} almost $\Delta$-regressive functions necessarily have large
monochromatic sets, regardless of cardinal arithmetic.
We consider $2^\kappa$ as a topological space with the ${<}\kappa$-supported
product topology, i.e., basic open sets are of the form $N_s :=
\{x \in 2^\kappa \mid s \subseteq x\}$ for $s \in 2^{<\kappa} = \bigcup_{\alpha <
\kappa} 2^\alpha$,
and we consider $[2^\kappa]^2$ as a topological space inheriting the
subspace topology from the product space $2^\kappa \times 2^\kappa$.
Recall that, in the space $2^\kappa$, the collection
of \emph{$\lambda$-Borel} sets, where $\lambda \leq \kappa$ is an infinite
cardinal, is the smallest collection of sets containing the open sets and
closed under complementation and unions and intersections of size $\lambda$.
Since we will only be working with $\kappa$-Borel sets, we will simply use
the word ``Borel'' to mean ``$\kappa$-Borel''. A subset of $2^\kappa$ is
\emph{meagre} if it is the union of $\kappa$-many nowhere dense sets.
Then it is readily established that $2^\kappa$ satisfies a version of the
Baire Category Theorem, i.e., it is
not the union of $\kappa$-many meagre sets.
Also, Borel subsets of $2^\kappa$ have the Baire property:
if $Y \subseteq 2^\kappa$ is Borel, then there is an open set
$U \subseteq 2^\kappa$ such that $Y \triangle U$ is meagre.
All of these comments carry over to the space $[2^\kappa]^2$, as well.
(See \cite{generalized_dst}, particularly Chapter 4, for these facts and
more on the descriptive set theory of $2^\kappa$ and $\kappa^\kappa$.)

We also consider $\kappa$ as a topological space with the discrete topology.
With these assumptions, note that $\Delta_\kappa : [2^\kappa]^2 \rightarrow
\kappa$ is continuous and thus Borel. The following theorem indicates a way
in which $\Delta_\kappa$ is provably a minimal Borel coloring with no
monochromatic sets of size $\kappa$.

\begin{thm}
  Suppose that $\kappa$ is an uncountable regular cardinal and that $c:[2^\kappa]^2
  \rightarrow \kappa$ is almost $\Delta$-regressive and Borel. Then there are
  $c$-monochromatic sets of size $\kappa$.
\end{thm}

\begin{proof}
  Let $\mu < \kappa$ witness that $c$ is almost $\Delta$-regressive.
  By recursion on $\xi$, we will construct sequences $(x_\xi)_{\xi < \kappa}$,
  $(\nu_\xi)_{\xi < \kappa}$, $(s_\xi)_{\xi < \kappa}$, $(Y_\xi)_{\xi < \kappa}$,
  and $(i_\xi)_{\xi < \kappa}$ such that
  \begin{itemize}
    \item for all $\xi < \kappa$, $x_\xi \in 2^\kappa$;
    \item $(\nu_\xi)_{\xi < \kappa}$ is an increasing, continuous sequence of
      infinite ordinals below $\kappa$;
    \item $(s_\xi)_{\xi < \kappa}$ is a $\subseteq$-increasing sequence of elements
      of $2^{<\kappa}$;
    \item for all $\xi < \kappa$, we have $\nu_\xi < |s_\xi| < |s_\xi| + 1 = \nu_{\xi + 1}$,
      $s_\xi \restriction \nu_\xi = x_\xi \restriction \nu_\xi$, and
      $s_\xi(\nu_\xi) \neq x_\xi(\nu_\xi)$;
    \item for all $\xi < \kappa$, $Y_\xi$ is a co-meagre subset of $N_{s_\xi}$ and,
      for all $y \in Y_\xi$, we have $c(x_\xi, y) = i_\xi$;
    \item for all $\xi < \eta < \kappa$, we have $x_\eta \in Y_\xi$.
  \end{itemize}
  Begin by letting $x_0$ be an arbitrary element of $2^\kappa$ and letting
  $\nu_0 = \mu$. Define $t_0 \in 2^{\mu + 1}$ by setting $t_0 \restriction \mu =
  x_0 \restriction \mu$ and $t_0(\mu) = 1 - x_0(\mu)$. By the
  Baire Category Theorem applied to $N_{t_0}$, we can find $i_0 < \kappa$
  such that $Y^*_0 := \{y \in N_{t_0} \mid c(x_0, y) = i_0\}$ is non-meagre.
  Since $c$ is Borel, $Y^*_0$ is a Borel subset of $2^\kappa$. Since Borel sets
  have the Baire property and $Y^*_0$ is a non-meagre Borel set, there is
  $s_0 \in 2^{<\kappa}$ such that $Y^*_0$ is co-meagre in $N_{s_0}$. Note
  that $s_0$ extends $t_0$, so $s_0 \restriction \nu_0 = x_0 \restriction \nu_0$ and
  $s_0(\nu_0) \neq x_0(\nu_0)$. Set $Y_0 = Y^*_0 \cap N_{s_0}$.

  Suppose next that $\eta < \kappa$ and that $(x_\xi)_{\xi \leq \eta}$, $(\nu_\xi)_{\xi
  \leq \eta}$, $(s_\xi)_{\xi \leq \eta}$, $(Y_\xi)_{\xi \leq \eta}$, and $(i_\xi)_{\xi \leq \eta}$ have
  been constructed, satisfying the requirements listed above. For
  all $\xi \leq \eta$, $Y_\xi$ is co-meagre in $N_{s_\eta}$, and hence
  $\bigcap_{\xi \leq \eta} Y_\xi$ is co-meagre in $N_{s_\eta}$. Let $x_{\eta + 1}$
  be an arbitrary element of $\bigcap_{\xi \leq \eta} Y_\xi$. (Note that
  $x_{\eta + 1} \in N_{s_\eta}$, since $Y_\eta \subseteq N_{s_\eta}$.) Let
  $\nu_{\eta + 1} = |s_\eta| + 1$, and define $t_\eta \in 2^{\nu_{\eta + 1}}$
  by letting $t_\eta \restriction |s_\eta| = s_\eta$ and $t_\eta(|s_\eta|) =
  1 - x_{\eta + 1}(|s_\eta|)$. Apply the Baire Category Theorem to $N_{t_\eta}$
  to find $i_{\eta + 1} < \kappa$ such that $Y^*_{\eta + 1} := \{y \in N_{t_\eta}
  \mid c(x_{\eta + 1}, y) = i_{\eta + 1}\}$ is non-meagre. Fix an $s_{\eta + 1}$
  such that $Y^*_{\eta + 1}$ is co-meagre in $N_{s_{\eta + 1}}$, and set
  $Y_{\eta + 1} = Y^*_{\eta + 1} \cap N_{s_{\eta + 1}}$.

  Finally, suppose that $\eta < \kappa$ is a limit ordinal and that $(x_\xi)_{\xi < \eta}$, $(\nu_\xi)_{\xi
  < \eta}$, $(s_\xi)_{\xi < \eta}$, $(Y_\xi)_{\xi < \eta}$, and $(i_\xi)_{\xi < \eta}$ have
  been constructed. Let $\nu_\eta = \sup\{\nu_\xi \mid \xi < \eta\}$, let $t^*_\eta = \bigcup_{\xi < \eta}
  s_\eta$, and note that $t^*_\eta \in 2^{\nu_\eta}$. For all $\xi < \eta$,
  $Y_\xi$ is co-meagre in $N_{t^*_\eta}$, so we can let $x_\eta$ be an arbitrary
  element of $N_{t^*_\eta} \cap \bigcap_{\xi < \eta} Y_\xi$. Define $t_\eta \in 2^{\nu_\eta + 1}$
  by letting $t_\eta \restriction \nu_\eta = t^*_\eta$ and $t_\eta(\nu_\eta) =
  1 - x_\eta(\nu_\eta)$. Now proceed exactly as in the previous case to define
  $i_\eta$, $s_\eta$, and $Y_\eta$. This concludes the construction.

  The point of our construction was to arrange so that, for all $\xi < \eta < \kappa$,
  we have $\Delta(x_\xi, x_\eta) = \nu_\xi$ and $c(x_\xi, x_\eta) = i_\xi$.
  Since $c$ is almost $\Delta$-regressive, it follows that the mapping
  $\nu_\xi \mapsto i_\xi$ is regressive, so, since $\{\nu_\xi \mid \xi < \kappa\}$
  is a club in $\kappa$ and hence stationary, we can apply the pressing-down lemma
  to find an unbounded $I \subseteq \kappa$ and a fixed $i < \kappa$ such that
  $i_\xi = i$ for all $\xi \in I$. In turn, $\{x_\xi \mid \xi \in I\}$ is a
  $c$-monochromatic set of size $\kappa$.
\end{proof}

Note that the above argument can be used to generate a monochromatic set that
has lexicographic order type $\kappa+1$.
We do not know how far this can be generalized, even for continuous almost
$\Delta$-regressive colourings.

\section{Maximal odd-cycle and triangle-free colourings}
\label{maximal_colouring_section}

\begin{dfn}
    Suppose that $X$ is a set and $\kappa$ is a cardinal. A colouring
    $c:[X]^2 \rightarrow \kappa$ is a \emph{maximal triangle-free colouring
    into $\kappa$} if
    \begin{enumerate}
        \item $c$ has no monochromatic triangles; and
        \item for any proper superset $Y \supsetneq X$ and any extension
        of $c$ to $c':[Y]^2 \rightarrow \kappa$, $c'$ does have a monochromatic
        triangle.
    \end{enumerate}
    The definition generalizes in the obvious way to \emph{$k$-cycle free},
    \emph{odd-cycle-free}, etc.
\end{dfn}

The following proposition is immediate.

\begin{prop} \label{triangle_free_equiv_prop}
  Suppose that $c:[X]^2 \rightarrow \kappa$ is a triangle-free colouring. The
  following statements are equivalent.
  \begin{enumerate}
      \item $c$ is a maximal triangle-free colouring.
      \item For every function $d:X \rightarrow \kappa$, there are distinct
      $x,y \in X$ such that $d(x) = d(y) = c(x,y)$.
  \end{enumerate}
\end{prop}

By the Erd\H{o}s-Rado theorem, for infinite $\kappa$,
if $c:[X]^2 \rightarrow \kappa$ is
triangle-free (or odd-cycle-free, $k$-cycle-free, etc.), then it must be
the case that $|X| \leq 2^\kappa$. Therefore, there must be maximal
triangle-free (or odd-cycle-free, etc.) colourings into $\kappa$ of size
at most $2^\kappa$; it turns out we have already seen an example of
such a colouring of size exactly $2^\kappa$.

For any infinite cardinal $\kappa$, the colouring $\Delta_\kappa:[2^\kappa]^2
\rightarrow \kappa$ is an odd-cycle-free colouring. In fact, it provides an
example of a \emph{maximal} odd-cycle free colouring and, indeed, a maximal
$k$-cycle free colouring for each odd $k \geq 3$. We provide a proof of this fact
for $k = 3$; an easy modification will work for odd $k > 3$, and we leave this
to the reader.

\begin{prop}
  Suppose that $\kappa$ is an infinite cardinal. Then $\Delta_\kappa$ is a
  maximal triangle-free colouring into $\kappa$.
\end{prop}

\begin{proof}
    Suppose not. Then, by Proposition \ref{triangle_free_equiv_prop},
    there is a function $d:2^\kappa \rightarrow \kappa$ such that, for
    all distinct $x,y \in 2^\kappa$, it is not the case that
    $d(x) = d(y) = \Delta(x,y)$.

    We will now construct an element $z \in 2^\kappa$ such that, for all
    $\alpha < \kappa$, there is no $x \in 2^\kappa$ such that
    $x \restriction (\alpha + 1) = z \restriction (\alpha + 1)$ and
    $d(x) = \alpha$. This will immediately result in a contradiction, because
    if $\alpha = d(z)$, then we clearly have $z \restriction (\alpha + 1)
    = z \restriction (\alpha + 1)$ and $d(z) = \alpha$.

    We will construct $z$ by specifying $z \restriction \alpha$ by recursion
    on $\alpha < \kappa$. To this end, fix $\alpha < \kappa$ and suppose
    that we have constructed $z \restriction \alpha$. We claim that there is
    at most one $i < 2$ for which there exists $x \in 2^\kappa$ such that
    $x \restriction (\alpha + 1) = (z \restriction \alpha) ^\frown \langle i \rangle$ and $d(x) = \alpha$.
    Indeed, otherwise there would be $x_0, x_1 \in 2^\kappa$ such that
    \begin{itemize}
        \item $x_0 \restriction \alpha = x_1 \restriction \alpha =
        z \restriction \alpha$;
        \item $x_0(\alpha) = 0$ and $x_1(\alpha) = 1$;
        \item $d(x_0) = d(x_1) = \alpha$.
    \end{itemize}
    But, in this case, we would have $\Delta(x_0, x_1) = \alpha =
    d(x_0) = d(x_1)$, contradicting our assumptions about $d$. Therefore,
    we can choose $i < 2$ such that there is no $x \in 2^\kappa$ with
    $x \restriction (\alpha + 1) = (z \restriction \alpha)^\frown \langle i
    \rangle$ and $d(x) = \alpha$, and then set $z(\alpha) = i$. This completes the construction of
    $z$ and thus the proof of the proposition.
\end{proof}

A natural question to ask now is the following: If $c:[X]^2 \rightarrow \kappa$
is a maximal triangle-free (or odd-cycle-free, etc.) colouring into $\kappa$,
must it be the case that $|X| = 2^\kappa$? For the case of odd-cycle-free colourings,
we have an affirmative answer. To help us prove this, let us introduce the following
notion.

\begin{dfn}
    Suppose that $X \subseteq 2^\kappa$ and $c:[X]^2 \rightarrow \kappa$.
    We say that $c$ is a $\delta$-colouring if, for all distinct $x,y \in X$,
    we have $x(c(x,y)) \neq y(c(x,y))$.
\end{dfn}

Notice the following relevant facts about $\delta$-colourings, which are
easily verified:
\begin{itemize}
    \item All $\delta$-colourings are odd-cycle-free.
    \item If $c:[X]^2 \rightarrow \kappa$ is a $\delta$-colouring, then
    $c$ can be extended to a $\delta$-colouring $c':[2^\kappa]^2 \rightarrow
    \kappa$ by letting $c'(x,y) = \Delta(x,y)$ for all
    $\{x,y\} \in [2^\kappa]^2 \setminus [X]^2$.
\end{itemize}

\begin{prop} \label{delta_colouring_prop}
  Suppose that $X$ is a set and $c:[X]^2 \rightarrow \kappa$ is odd-cycle-free.
  Then $c$ is isomorphic to a $\delta$-colouring. In other words, there is
  an injective function $\iota:X \rightarrow 2^\kappa$ such that that the
  function $c':[\iota``X]^2 \rightarrow \kappa$ defined by
  $c'(\iota(x), \iota(y)) = c(x, y)$ is a $\delta$-colouring.
\end{prop}

\begin{proof}
    The fact that $c$ is odd-cycle-free is equivalent to the assertion that,
    for each $\alpha < \kappa$, the graph $G_\alpha = (X, c^{-1}(\{\alpha\}))$
    is bipartite. Therefore, for each $\alpha < \kappa$, we can partition
    $X$ into two sets $X = X^\alpha_0 \dot{\cup} X^\alpha_1$ such that
    for all $i < 2$ and distinct $x, y \in X^\alpha_i$, we have
    $c(x,y) \neq \alpha$. For each $x \in X$ and $\alpha < \kappa$,
    let $i_\alpha(x)$ be the unique $i < 2$ such that $x \in X^\alpha_i$,
    and define $\iota(x) \in 2^\kappa$ by letting $\iota(x)(\alpha) = i_\alpha(x)$
    for all $\alpha < \kappa$.

    \begin{clm}
        $\iota$ is injective.
    \end{clm}

    \begin{proof}
        Fix distinct $x,y \in X$, and let $\alpha = c(x,y)$. Then it must
        be the case that $i_\alpha(x) \neq i_\alpha(y)$, so $\iota(x) \neq
        \iota(y)$.
    \end{proof}

    To finish the proof, it suffices to show that the colouring $c'$ in the
    statement of the proposition is a $\delta$-colouring. To see this,
    fix distinct $x, y \in X$ and let $\alpha = c(x,y) = c'(\iota(x), \iota(y))$.
    Then, by construction, $i_\alpha(x) \neq i_\alpha(y)$, so
    $\iota(x)(\alpha) \neq \iota(y)(\alpha)$, so $c'$ is in fact a
    $\delta$-colouring.
\end{proof}

\begin{cor}
    If $c:[X]^2 \rightarrow \kappa$ is a maximal odd-cycle-free colouring into
    $\kappa$, then $|X| = 2^\kappa$.
\end{cor}

\begin{proof}
    By Erd\H{o}s-Rado, we know that $|X| \leq 2^\kappa$. For the other
    inequality, apply Proposition \ref{delta_colouring_prop} to find an
    injection $\iota:X \rightarrow 2^\kappa$ such that the colouring
    $c':[\iota``X]^2 \rightarrow \kappa$ defined by $c'(\iota(x), \iota(y)) =
    c(x,y)$ is a $\delta$-colouring. If $|X| < 2^\kappa$, then $\iota``X$
    is a proper subset of $2^\kappa$, so $c'$ can be properly extended to
    a $\delta$-colouring $d':[2^\kappa]^2 \rightarrow \kappa$. But then
    $d'$ easily induces a proper extension of $c$ to an odd-cycle-free
    colouring $c':[X \cup (2^\kappa \setminus \iota``X)]^2 \rightarrow \kappa$,
    contradicting the fact that $c$ is a maximal odd-cycle-free colouring into
    $\kappa$.
\end{proof}

The analogous question about maximal triangle-free colourings remains open.
The simplest case of this question asks whether there is always a
maximal triangle-free colouring $c:[\omega_1]^2 \rightarrow \omega$, or
even whether there is consistently such a colouring in a model of $\neg \mathrm{CH}$.
One way to ensure that a colouring $c:[\omega_1]^2 \rightarrow \omega$
is triangle-free is to require that all of the fibers
$c(\cdot, \beta)$ be one-to-one. Such, colourings, however, necessarily
fail to be maximal.

\begin{prop} \label{one_to_one_prop}
  Suppose that $c:[\omega_1]^2 \rightarrow \omega$ has the property
  that, for all $\beta < \omega_1$, the map $c(\cdot, \beta):\beta
  \rightarrow \omega$ is one-to-one. Then $c$ is not a maximal
  triangle-free colouring.
\end{prop}

\begin{proof}
    Suppose that $c$ is triangle-free.
    To show that $c$ is not maximal, it suffices to construct a
    function $d:\omega_1 \rightarrow \omega$ such that, for all
    $\alpha < \beta < \omega_1$, it is not the case that $d(\alpha)
    = d(\beta) = c(\alpha, \beta)$. To begin, fix a colour $i < \omega$
    such that $i \neq c(0, 1)$, and let $d(0) = d(1) = i$. Next,
    for each $1 < \alpha < \omega_1$, there must be $\epsilon_\alpha
    < 2$ such that $c(\epsilon_\alpha, \alpha) \neq i$. Let $d(\alpha)
    = c(\epsilon_\alpha, \alpha)$.

    To verify that $d$ is as desired, fix $\alpha < \beta < \omega_1$.
    If $\alpha = 0$ and $\beta = 1$, then $d(\alpha) = d(\beta) = i$
    and $c(\alpha, \beta) \neq i$, so the requirement is
    satisfied. If $\alpha < 2$ and $\beta \geq 2$, then $d(\alpha) = i$
    and $d(\beta) \neq i$, so again the requirement is satisfied.
    Finally, if $2 \leq \alpha$, then $d(\beta) = c(\epsilon_\beta,
    \beta) \neq c(\alpha, \beta)$, since $c(\cdot, \beta)$ is
    injective, so the requirement is satisfied once again.
\end{proof}

At this point, it is unclear whether Proposition \ref{one_to_one_prop}
can be strengthened to apply to maps with finite-to-one fibers. We do,
however, have the following result.

\begin{thm} \label{finite-to-one_forcing_thm}
  Suppose that $c : [\omega_1]^2 \rightarrow \omega$ has the property that,
  for all $\beta < \omega_1$, the map $c(\cdot, \beta):\beta \rightarrow \omega$
  is finite-to-one. Then there is a ccc forcing notion $\bb{P}$ such that
  \[
    \Vdash_{\bb{P}} ``\check{c} \text{ is not a maximal triangle-free colouring}".
  \]
\end{thm}

\begin{proof}
  Suppose that $c$ is triangle-free. Our forcing notion $\bb{P}$ will consist of
  finite attempts to extend the colouring $c$. More precisely, conditions of
  $\bb{P}$ are pairs $p = (s^p, f^p)$ such that
  \begin{itemize}
    \item $s^p \in [\omega_1]^{<\omega}$;
    \item $f^p : s^p \rightarrow \omega$;
    \item for all $\alpha < \beta$ in $s^p$, it is not the case that
    $f^p(\alpha) = f^p(\beta) = c(\alpha, \beta)$.
  \end{itemize}
  If $p,q \in \bb{P}$, then $q \leq_{\bb{P}} p$ if $s^q \supseteq s^p$ and
  $f^q \supseteq f^p$.

  \begin{clm}
    $\bb{P}$ has the ccc.
  \end{clm}

  \begin{proof}
    Suppose for sake of contradiction that $\mathcal{A} =
    \{p_\eta \mid \eta < \omega_1\}$ is an antichain in $\bb{P}$.
    For $\eta < \omega_1$, let $s^\eta = s^{p_\eta}$ and $f^\eta = f^{p_\eta}$.
    By thinning out $\mathcal{A}$ if necessary, we can assume that the sets
    $\{s^\eta \mid \eta < \omega_1\}$ form a head-tail-tail $\Delta$-system
    with root $r$. More precisely, for all $\eta < \xi < \omega_1$, we have
    \begin{itemize}
      \item $s^\eta \cap s^\xi = r$; and
      \item $r < s^\eta \setminus r < s^\xi \setminus r$.
    \end{itemize}
    By thinning out further, we can also assume that there is a single function
    $g:r \rightarrow \omega$ such that $f^\eta \restriction r = f$ for all
    $\eta < \omega_1$.

    It follows that, for all $\eta < \xi < \omega_1$, $f^\eta \cup f^\xi$ is
    a function. Let $q_{\eta \xi} = (s^\eta \cup s^\xi, f^\eta \cup f^\xi)$.
    If $q_{\eta \xi}$ were a condition in $\bb{P}$, then it would be a common
    lower bound to $p_\eta$ and $p_\xi$, contradicting the assumption that
    $\mathcal{A}$ is an antichain. Therefore, by the definition of $\bb{P}$,
    there must be
    $\alpha_{\eta \xi} \in s^\eta \setminus r$ and $\beta_{\eta \xi} \in
    s^\xi \setminus r$ such that $f^\eta(\alpha_{\eta \xi}) = f^\xi(\beta_{\eta
    \xi}) = c(\alpha_{\eta \xi}, \beta_{\eta \xi})$.

    Now, if $\omega \leq \xi < \omega_1$, there must be a fixed
    $\beta_\xi \in s^\xi \setminus r$ such that the set
    $X = \{\eta < \xi \mid \beta_{\eta \xi} = \beta_\xi\}$ is infinite.
    But then, for all $\eta \in X$, we have
    $c(\alpha_{\eta \xi}, \beta_\xi) = f^\xi(\beta_\xi)$ and,
    for all $\eta < \eta'$ in $X$, we have $\alpha_{\eta \xi} <
    \alpha_{\eta' \xi} < \beta$. Therefore, $c(\cdot, \beta)$ is not
    finite-to-one, with the failure witnessed by the colour $f^\xi(\beta_\xi)$
    and the infinite set $\{\alpha_{\eta \xi} \mid \eta \in X\}$. This is
    a contradiction.
  \end{proof}
  For each $\alpha < \omega_1$, let $D_\alpha = \{p \in \bb{P} \mid
  \alpha \in s^p\}$. It is easily verified that $D_\alpha$ is dense in
  $\bb{P}$ for all $\alpha < \omega_1$. Thus, if $G$ is $\bb{P}$-generic over
  $V$, then $f_G = \bigcup \{f^p \mid p \in G\}$ is a function from
  $\omega_1$ to $\omega$ such that, for all $\alpha < \beta < \omega_1$,
  it is not the case that $f_G(\alpha) = f_G(\beta) = c(\alpha, \beta)$.
  By Proposition \ref{triangle_free_equiv_prop}, it follows that $c$ is not
  a maximal triangle-free colouring in $V[G]$.
\end{proof}

\begin{cor}
  If $\mathrm{MA}_{\omega_1}$ holds, then there are no maximal triangle-free
  colourings $c:[\omega_1]^2 \rightarrow \omega$ with finite-to-one fibers.
\end{cor}

\begin{proof}
  Assume that $\mathrm{MA}_{\omega_1}$ holds, and fix a triangle-free
  colouring $c:[\omega_1]^2 \rightarrow \omega$ with finite-to-one fibers.
  Apply $\mathrm{MA}_{\omega_1}$ to the poset $\bb{P}$ and the dense sets
  $\{D_\alpha \mid \alpha < \omega_1\}$ isolated in the proof of Theorem
  \ref{finite-to-one_forcing_thm} to obtain a function $f:\omega_1 \rightarrow
  \omega$ witnessing that $c$ is not maximal.
\end{proof}

Though we do not know of the consistency of a maximal triangle-free colouring
$c:[\omega_1]^2 \rightarrow \omega$ in a model of $\neg \mathrm{CH}$,
we can arrange the consistency of the existence of a a maximal triangle-free
colouring of some proper subset of $[\omega_1]^2$. Here, we say that
an $\omega$-colouring $F$ whose domain is a subset of $[\omega_1]^2$
is maximal if, for every $d:\omega_1 \rightarrow \omega$, there are
$\alpha < \beta < \omega_1$ such that $\{\alpha, \beta\} \in \mathrm{dom}(F)$
and $d(\alpha) = d(\beta) = F(\alpha, \beta)$.

\begin{thm}
There is a ccc poset $\mb P$ of size $\aleph_1$ such that
\[
  \force_{\mb P} ``\text{There is a maximal monochromatic triangle-free }\omega\text{-colouring }F\text{ with }\dom F\subs [\omg]^2".
\]
\end{thm}

In particular, the continuum can be arbitrary large. However, we lack  techniques to define a maximal $F$ on all of $[\omg]^2$.

\begin{proof}
For each $\delta\in \lim \omg$, pick $\vareps_\delta<\delta$ so that the set $\{\delta \mid \vareps_\delta = \vareps\}$ is stationary for all $\vareps$.

Let $\mb P$ consist of all $p=(s^p,(g^p_k)_{k\in n^p})$ so that
\begin{enumerate}
    \item $s^p\in [\omg]^{<\oo}$, $n^p\in \oo$, $g^p_k\subs [s^p]^2$;
    \item  $g^p_k$ is triangle-free;
    \item  $g^p_k\cap g^p_\ell=\emptyset$ for all $k<\ell<n^p$; and
    \item for $\delta' < \delta$, if $\delta'\delta\in g^p_k$, then
    $\vareps_\delta\leq \delta'$.
\end{enumerate}

If $H\subs \mb P$ is a generic filter, we let $F(\delta'\delta)=k$ for some $\delta'<\delta$ if there is  $p\in H$ so that $k<n^p$ and $\delta'\delta\in g^p_k$.\footnote{So, in a  condition $p$, $g^p_k$ approximates the colour class $F^{-1}(k)$.}

Let us show that $F$ is maximal, which will also imply that $\omg$ is not collapsed.

\begin{clm}
For any partition $\omg=\bigcup_{k\in \oo}X_k$ there is some $k<\oo$ and $\delta'<\delta\in X_k$ so that $F(\delta'\delta)=k$.
\end{clm}
\begin{proof}
  Working back in $V$, fix $\bb{P}$-names $(\dot{X}_k)_{k \in \omega}$
  for $(X_k)_{k \in \omega}$, and fix a condition $p$ such that $p\force ``\omg=\bigcup_{k\in \oo}\dot X_k"$. Take a continuous, increasing sequence of elementary submodels $(M_\alpha)_{\alpha<\omg}$ of
some sufficiently large $H(\theta)$ so that $p,(\dot X_k)_{k \in \omega}
\in M_0$.

Let $\vareps=\omg \cap M_0$ and find $\alpha<\omg$ so that $\vareps=\vareps_{\delta}$ where $\delta=M_\alpha\cap \omg$. Now, find some $q\leq p$ and $k\in \oo$ so that $\delta \in s^q$ and $q\force ``\delta\in \dot X_k"$. By extending $q$ further, we can assume that $k<n^q$.

We can write $s^q$ as the union of the three sets $s_0<s_1<s_2$ where $s_0=s^q\cap M_0$, $s_1=s^q\cap M_\alpha\setm M_0$, and $s_2 = s^q \setminus M_\alpha$.

\begin{subclaim}
There is a $q'\in \mb P$ with $s^{q'}=s_0\cup s_1'\cup s_2'$ with $\delta'=\min s_2'$ and
\begin{enumerate}[(i)]
    \item  $q$ and $q'$ are isomorphic,
    \item $s_0<s_1'\subs M_0$, $s_1<s_2'\subs M_\alpha$,
    \item $q'\force ``\delta'\in \dot X_k"$.
\end{enumerate}

\end{subclaim}

The proof is a double reflection argument using elementarity.
\medskip

Now, we can define a condition $r$ that extends both $q$ and $q'$ such
that $r\force ``\dot{F}(\delta'\delta)=k"$. Indeed, we let $s^r=s^q\cup s^{q'}$, $n^r=n^q=n^{q'}$ and $g^r_\ell=g^q_\ell\cup g^{q'}_\ell$ for $k\neq \ell<n^q$ and let $g^r_k=g^q_k\cup g^{q'}_k\cup\{\delta'\delta\}$.

The only way that $r$ can fail to be a condition in $\bb{P}$ is if $g^r_k$ is not triangle-free. However, any triangle in $g^r_k$ must contain the new edge $\delta'\delta$ and their common neighbour must lie in $s_0$. However $s_0\subs \vareps_\delta$ while any neighbour of $\delta$ is at least $\vareps_\delta$.
\end{proof}

Proving the ccc of $\mb P$ is very similar but we don't even need to add any edge when amalgamating isomorphic conditions.
\end{proof}

Note that in the above proof, we proved that each colour class of $F$ has uncountable chromatic number.

\section{Further remarks on triangle-free colourings}
\label{further_remarks_section}

In this section, we prove some further results about triangle-free colourings
on $[\omega_1]^2$ motivated by the question from the previous section about
whether there necessarily exists a maximal triangle-free colouring
$c:[\omega_1]^2 \rightarrow \omega$. This can be seen as a specific instance of
a more general question: Do there exist triangle-free colourings
$c:[\omega_1]^2 \rightarrow
\omega$ in which all colour classes $c^{-1}(\{k\})$ are ``large''? This
leads naturally to the consideration of square bracket partition relations,
whose definition we now recall.

\begin{dfn}
  For cardinals $\kappa, \lambda, \mu$, and $\theta$, the partition relation
  $\kappa \rightarrow [\lambda]^\theta_\mu$ is the assertion that, for every
  $f:[\kappa]^\theta \rightarrow \mu$, there is $X \in [\kappa]^\lambda$ such
  that $f``[X]^\theta \neq \mu$, i.e., $f$ omits at least one colour on
  $[X]^\theta$. If $\kappa$ is a regular cardinal, then the relation
  $\kappa \rightarrow [Stat]^\theta_\mu$ is obtained by replacing the
  requirement $X \in [\kappa]^\lambda$ above with the requirement that
  $X$ is stationary in $\kappa$.
\end{dfn}

Much work has been done analyzing colourings on $[\omg]^2$ witnessing the
failure of square bracket partition relations, with the most notable result being
Todor\v{c}evi\'{c}'s proof of $\omega_1 \not\rightarrow [\omega_1]^2_{\omega_1}$
in \cite{todorcevic_partitioning}. Recall from the previous section that, if
$c:[\omega_1]^2 \rightarrow \omega$ is triangle-free, then it is a \emph{maximal}
triangle-free colouring into $\omega$ if, for every function
$d:\omega_1 \rightarrow \omega$, there are $\alpha < \beta < \omega_1$ such
that $d(\alpha) = d(\beta) = c(\alpha, \beta)$. This latter condition is
easily seen to be satisfied if $c$ witnesses $\omega_1 \not\rightarrow
[\omega_1]^2_\omega$. However, the following easy observation shows that
such colourings can never be triangle-free.

\begin{obs}
Suppose that $c:[\omg]^2\to \oo$ witnesses $\omg\not\to[Stat]^2_\oo$. Then $c$
has infinite monochromatic sets in all colours.
\end{obs}
\begin{proof}
  This follows from the Dushnik-Miller relation: $\omg\to(Stat,\oo)^2_2$, i.e.,
  for every function $f:[\omega_1]^2 \rightarrow 2$, there is either a
  stationary $X \subseteq \omega_1$ such that $f``[X]^2 = \{0\}$ or there is
  an infinite $Y \subseteq \omega_1$ such that $f``[Y]^2 = \{1\}$. If we could not
  find an infinite monochromatic set of colour $k$ then there would be a
  stationary set that omits colour $k$, which contradicts the assumption on $c$.
\end{proof}

On the other hand, the following holds, where $\mf s$ denotes the
\emph{splitting number}.\footnote{Unpublished result from personal communication.}

\begin{thm}[D. Raghavan]
  Suppose that $\mf s=\aleph_1$. Then there is a
  triangle-free $c:[\omg]^2\to \oo$ so that for any uncountable $X\subset \omg$, $c``[X]^2$
  is co-finite.
\end{thm}

We wonder if the conclusion of this theorem holds in ZFC.

\medskip

One way in which a subset of $[\omega_1]^2$ can be considered ``large" is
by having uncountable chromatic number as a graph. It is easily seen
that there are always triangle-free colourings $c:[\omega_1]^2 \rightarrow
\omega$ for which all colour classes have uncountable chromatic number.

\begin{prop}
  There is a colouring  $c:[\omg]^2\to \oo$ so that each colour class $G_n=c^{-1}(n)$ is triangle-free of uncountable chromatic number.
\end{prop}
\begin{proof}
  Let $\omg=\bigcup\{S_n:n<\oo\}$ be a partition into uncountable sets, and, for each $n < \omega$, let $H_n$ be a triangle-free graph of uncountable chromatic number with vertex set $S_n$ (for an example of such a graph, see \cite{erdos_rado}). Now, define $c$ so that $c(\alpha,\beta)=n$ if $\alpha\beta\in H_n$ and so that for any $\beta \in S_n$,
  \[
    c(\cdot,\beta)\uhr\{\alpha<\beta:\alpha\beta\notin\bigcup_{n<\oo}H_n\}
  \]
  is injective and maps into $\omega \setminus \{n\}$. It is easy to see that
  $c$ satisfies our requirements.
\end{proof}

Next, we prove that using some additional assumptions, we can make each colour class $G_n$ quite thin.
If $G \subseteq [\omega_1]^2$ is a graph and $\alpha < \omega_1$, then we let
$G(\alpha) = \{\beta < \alpha \mid \{\beta, \alpha\} \in G\}$. Recall that a
graph $G \subseteq [\omega_1]^2$ is a \emph{Hajnal-M\'{a}t\'{e} graph} if, for
every $\alpha < \omega_1$, $G(\alpha)$ is either a finite set or an
$\omega$-sequence converging to $\alpha$.

The existence of Hajnal-M\'{a}t\'{e} graphs with uncountable chromatic number
turns out to be independent of ZFC. In \cite{hajnal_mate}, Hajnal and
M\'{a}t\'{e} prove that $\diamondsuit^+$ implies the existence of
Hajnal-M\'{a}t\'{e} graphs with uncountable chromatic number, while Martin's
Axiom, $\mathrm{MA}_{\aleph_1}$, implies that every Hajnal-M\'{a}t\'{e} graph
has countable chromatic number. Komj\'{a}th, in \cite{komjath_note_on_hajnal_mate},
improves upon the first result by proving that, if $\diamondsuit$ holds, then
there are triangle-free Hajnal-M\'{a}t\'{e} graphs with uncountable chromatic
number. We improve this result further with the following theorem.

\begin{thm}
  Suppose that $\diamondsuit$ holds. Then there is a partition $[\omega_1]^2 =
  \bigcup_{n < \omega} G_n$ such that each $G_n$ is a triangle-free Hajnal-M\'{a}t\'{e}
  graph with uncountable chromatic number.
\end{thm}

\begin{proof}
  Since $\diamondsuit$ holds, we can find pairwise disjoint stationary sets
  $\{S^n_\delta \mid n < \omega, ~ \delta < \omega_1\}$ such that $\diamondsuit(S^n_\delta)$
  holds for each $n < \omega$ and $\delta < \omega_1$. For each $n$ and $\delta$,
  we can assume that $S^n_\delta$ consists solely of limit ordinals greater than
  $\delta$, and we can fix a sequence $\langle f^n_{\delta, \alpha} : \alpha \rightarrow
  \omega \mid \alpha < \omega_1 \rangle$ such that, for every $f:\omega_1 \rightarrow
  \omega$, there are stationarily many $\alpha \in S^n_\delta$ for which
  $f^n_{\delta, \alpha} = f \restriction \alpha$.

  We are going to define a function $g:[\omega_1]^2 \rightarrow \omega$ such that,
  for all $n < \omega$, $G_n = g^{-1}\{n\}$ will be as desired. We will define
  $g \restriction [\alpha]^2$ by recursion on $\alpha < \omega_1$. For
  each $\alpha < \omega_1$, let $g_\alpha : \alpha \rightarrow \omega$ denote
  the function defined by letting $g_\alpha(\beta) = g(\beta, \alpha)$ for all
  $\beta < \alpha$.

  Suppose that $\alpha < \omega_1$ and that we have defined $g \restriction
  [\alpha]^2$. We now describe how to define $g \restriction [\alpha + 1]^2$,
  which amounts to defining $g_\alpha : \alpha \rightarrow \omega$. If there
  are no $n < \omega$ and $\delta < \omega_1$ such that $\alpha \in S^n_\delta$,
  then simply let $g_\alpha$ be an arbitrary injective function. Note that this
  introduces no triangles to any $G_n$ and ensures that $|G_n(\alpha)| = 1$
  for every $n < \omega$.

  Otherwise, let $n < \omega$ and $\delta < \omega_1$ be such that $\alpha \in
  S^n_\delta$. We first specify the set of $\beta < \alpha$ for which
  $g_\alpha(\beta) = n$. Begin by fixing an increasing $\omega$-sequence
  $\langle \alpha_k \mid k < \omega \rangle$ converging to $\alpha$ with
  $\delta = \alpha_0$. By recursion on $k < \omega$, we will construct a set
  $A_\alpha \subseteq \omega$ and a sequence $\langle \beta^\alpha_k \mid k \in A_\alpha \rangle$
  such that
  \begin{itemize}
    \item for all $k \in A_\alpha$, we have $\max\{\alpha_k, \max\{\beta^\alpha_j
      \mid j \in A_\alpha \cap k\}\} < \beta^\alpha_k < \alpha$;
    \item for all $k \in A_\alpha$, we have $f^n_{\delta, \alpha}(\beta^\alpha_k) = k$;
    \item for all $j < k$ in $A_\alpha$, we have $g(\beta^\alpha_{j}, \beta^\alpha_{k})
      \neq n$.
  \end{itemize}
  The construction is straightforward. If $k < \omega$ and we have specified
  $A_\alpha \cap k$ and $\{\beta^\alpha_j \mid j \in A_\alpha \cap k\}$, then ask whether there is
  $\beta$ such that $\max\{\alpha_k, \max\{\beta^\alpha_j \mid j \in A_\alpha \cap k\}\}
  < \beta < \alpha$, $f^n_{\delta, \alpha}(\beta) = k$, and, for all $j \in
  A_\alpha \cap k$, $g(\beta^\alpha_j, \beta) \neq n$. If there is, then put $k$ into
  $A_\alpha$ and let $\beta^\alpha_k$ be the least such $\beta$. Otherwise, leave
  $k$ out of $A_\alpha$ and leave $\beta^\alpha_k$ undefined.

  Now define $g_\alpha : \alpha \rightarrow \omega$ by first requiring that $G_n(\alpha) =
  \{\beta^\alpha_k \mid k \in A_\alpha\}$. Note that this set is either finite
  or an $\omega$-sequence converging to $\alpha$. Now define $g_\alpha$ on
  $\alpha \setminus G_n(\alpha)$ to be an injective function into $\omega
  \setminus \{n\}$. Also note that our construction adds no new triangles to any
  $G_n$.

  This finishes the construction of $g$. It is clear that each $G_n$ is a
  triangle-free Hajnal-M\'{a}t\'{e} graph. It remains to show that each $G_n$
  is uncountably chromatic. Suppose for sake of contradiction that $n < \omega$
  and $f:\omega_1 \rightarrow \omega$ is a proper colouring for $G_n$. For each
  $\delta < \omega_1$, we introduce the following notation.
  \begin{itemize}
    \item $T_\delta$ is the stationary set of $\alpha \in S^n_\delta$ for which
      $f^n_{\delta, \alpha} = f \restriction \alpha$.
    \item For all $k < \omega$, $T_{\delta, k}$ is the set of $\alpha \in
      T_\delta$ for which $f(\alpha) = k$.
  \end{itemize}
  It is easy to see that there must be $k < \omega$ such that, for unboundedly
  many $\delta  < \omega_1$, $T_{\delta, k}$ is stationary in $\omega_1$. Fix
  such a $k$. Let $T = \bigcup_{\delta < \omega_1} T_{\delta, k}$, and let
  $E = \{\delta < \omega_1 \mid T_{\delta, k} \mbox{ is unbounded in } \omega_1\}$.
  By our choice of $k$, $T$ is stationary and $E$ is unbounded in $\omega_1$.

  Using the normality of the club filter, we can find $\alpha \in T \cap \lim(E)
  \cap (\Delta_{\delta \in E} \lim(T_{\delta, k}))$. Let $\delta^* < \omega_1$
  be such that $\alpha \in S^n_{\delta^*}$. Since $f(\alpha) = k$,
  it must be the case that, in our construction of $g_\alpha$, we left
  $\beta^\alpha_k$ undefined, because otherwise we would have $f(\beta^\alpha_k) =
  f^n_{\delta^*, \alpha}(\beta^\alpha_k) = k = f(\alpha)$ and $\{\beta^\alpha_k,
  \alpha\} \in G_n$, contradicting the fact that $f$ is a good colouring for $G_n$.
  But now we can find $\delta \in E \cap \alpha$ such that $\beta^\alpha_j < \delta$
  for all $j \in A_\alpha \cap k$. By our choice of $\alpha$, we can find
  $\beta \in T_{\delta, k} \cap \alpha$ with $\beta > \alpha_k$. By our construction
  of $g_\beta$, it follows that, for every $\gamma < \delta$, $g_\beta(\gamma) \neq n$.
  It also follows that $f^n_{\delta^*, \alpha}(\beta) = f(\beta) = k$.
  But then it is easily seen that $\beta$ gives a positive answer to the question
  asked at stage $k$ of the construction of $A_\alpha$, in which case $\beta^\alpha_k$
  is in fact defined. This contradiction completes the proof.
\end{proof}

Note that the graphs $G_n$ defined in the proof above actually satisfy the following
strengthening of triangle-freeness: for all $3 \leq \ell < \omega$, there are no
cycles $\langle \alpha_0, \alpha_1, \ldots, \alpha_{\ell - 1} \rangle$ for which
$\alpha_0 < \alpha_1 < \ldots < \alpha_{\ell - 1}$.

\section{Open problems}
\label{problem_section}

In this final section, we collect some remaining open problems.
We start with the most important questions stemming directly from our investigations. First, on regressive and almost-regressive colourings we ask the following.
\medskip

\begin{prob}
  Suppose that $\kappa$ is a regular uncountable cardinal, $2^\mu < 2^\kappa$
  for every $\mu < \kappa$, and $c$ is an almost $\Delta$-regressive colouring
  on $[2^\kappa]^2$.
  \begin{enumerate}
    \item Does $c$ necessarily have monochromatic triples?
    \item Does $c$ necessarily have infinite monochromatic subsets?
    \item Does $c$ necessarily have monochromatic subsets of size $\kappa$?
  \end{enumerate}
\end{prob}

\begin{prob}
  Suppose that $\kappa$ is a regular uncountable cardinal and
  $c$ is a $\Delta$-regressive colouring on $[2^\kappa]^2$.
  Must there exist a $c$-monochromatic set of size $4$? What about
  an infinite $c$-monochromatic set? What about uncountable
  $c$-monochromatic sets?
\end{prob}

Regarding maximal triangle-free colourings, the next questions are the most
natural.

\begin{prob}
Is there a maximal triangle-free colouring $c:[\omg]^2\to \oo$?
\end{prob}

\begin{prob}Is there, consistently, a maximal triangle-free colouring of $[\omg]^2$ which embeds no uncountable $\delta$-colourings?

\end{prob}

\begin{prob} Assume $MA_{\aleph_1}$ (or even PFA).
 Suppose that $X\subset \mb R$ has size $\aleph_1$ and $c:[X]^2\to \oo$ is a continuous/Borel. Can $c$ be maximal triangle-free?
\end{prob}

It would also be natural to look at colouring triples and in general,
$[2^\kappa]^n$ for some finite $n<\oo$ and look for critical examples. We mention
the following results.

\begin{thm} [Todorcevic, \cite{todorcevic_combinatorial_cubes}]
  There is a colouring $c: [2^\omg]^3 \rightarrow 10$ such that all colours
  appear on any uncountable $X \subseteq 2^\omg$. More generally,
  for every $r \geq 3$, there is a colouring $c: [2^\omg]^r \rightarrow
  r!(r-1)! - 2$ such that all colours appear on any uncountable
  $X \subseteq 2^\omg$.
\end{thm}

\begin{prop}
  There is a colouring $c:[2^\omg]^{3}\to \omg$ such that all colours appear on
  any dense-in-itself $X\subseteq 2^\omg$.
\end{prop}

\begin{proof}
  Use a colouring witnessing $\omega_1 \not\rightarrow [\omega_1]^2_{\omega_1}$
  on the two $\Delta$-values determined by any triple in $[2^\omg]^3$.
\end{proof}

Next, we mention a prominent open problem concerning uncountable
triangle-free graphs that is tangentially related to our work.

\begin{prob}[Erd\H{o}s]
  Is there, in ZFC, a graph $G$ of uncountable chromatic number
  so that any triangle-free subgraph of $G$ has countable chromatic number?
\end{prob}

Komj\'{a}th and Shelah \cite{kopeshelah} proved that, consistently,
the answer is yes.
They also proved that if $\chi(G)\geq \mf c^+$ then either $G$ contains a $K_4$ or a triangle free subgraph of uncountable chromatic number.  We wonder if similarly to this and \cite{rodl} where the finite case is dealt with, one can show:

\begin{prob}
Suppose that $\kappa$ is strongly inaccessible, and let $G$ have chromatic number $\kappa$. Is there, for any $\lambda<\kappa$, a triangle-free subgraph of $G$ of chromatic number at least $\lambda$?
\end{prob}

Finally, we end with two questions about colourings $c:[\omega_1]^2 \rightarrow
\omega$ with large colour classes.

\begin{prob}
  Is there, in ZFC, a triangle-free  $c:[\omg]^2\to \oo$ so that for any
  uncountable $X\subset \omg$, $c\uhr [X]^2$ assumes all but finitely many colours.
\end{prob}

\begin{prob}
 Is there  a triangle-free  $c:[\omg]^2\to \oo$ so that for any partition $\omg=\bigcup_{i<\oo}X_i$, there is some $i<\oo$ so that $c\uhr [X_i]^2$ assumes all values.
\end{prob}

This last question seems closely related to the simultaneous chromatic number
problems studied by Hajnal and Komj\'{a}th \cite{simchrom}.


\begin{thebibliography}{10}

\bibitem{baumgartner_ineffability}
J.~E. Baumgartner.
\newblock Ineffability properties of cardinals. {I}.
\newblock In {\em Infinite and finite sets ({C}olloq., {K}eszthely, 1973;
  dedicated to {P}. {E}rd\H{o}s on his 60th birthday), {V}ol. {I}}, pages
  109--130. Colloq. Math. Soc. J\'{a}nos Bolyai, Vol. 10. 1975.

\bibitem{erdos_rado}
P.~Erd\H{o}s and R.~Rado.
\newblock A construction of graphs without triangles having pre-assigned order
  and chromatic number.
\newblock {\em J. London Math. Soc.}, s1-35(4):445--448, 1960.

\bibitem{erdos_rado_partition_calculus}
P.~Erd\"{o}s and R.~Rado.
\newblock A partition calculus in set theory.
\newblock {\em Bull. Amer. Math. Soc.}, 62:427--489, 1956.

\bibitem{generalized_dst}
Sy-David Friedman, Tapani Hyttinen, and Vadim Kulikov.
\newblock Generalized descriptive set theory and classification theory.
\newblock {\em Mem. Amer. Math. Soc.}, 230(1081):vi+80, 2014.

\bibitem{simchrom}
Andr\'{a}s Hajnal and P\'{e}ter Komj\'{a}th.
\newblock Some remarks on the simultaneous chromatic number.
\newblock {\em Combinatorica}, 23(1):89--104, 2003.

\bibitem{hajnal_mate}
Andr\'{a}s Hajnal and Attila M\'{a}t\'{e}.
\newblock Set mappings, partitions, and chromatic numbers.
\newblock In {\em Logic {C}olloquium '73 ({B}ristol, 1973)}, pages 347--379.
  North-Holland, Amsterdam, 1975.

\bibitem{komjath_note_on_hajnal_mate}
P\'{e}ter Komj\'{a}th.
\newblock A note on {H}ajnal-{M}\'{a}t\'{e} graphs.
\newblock {\em Studia Sci. Math. Hungar.}, 15(1-3):275--276, 1980.

\bibitem{komjath_second_note}
P\'{e}ter Komj\'{a}th.
\newblock A second note on {H}ajnal-{M}\'{a}t\'{e} graphs.
\newblock {\em Studia Sci. Math. Hungar.}, 19(2-4):245--246, 1984.

\bibitem{kopeshelah}
P\'{e}ter Komj\'{a}th and Saharon Shelah.
\newblock Forcing constructions for uncountably chromatic graphs.
\newblock {\em J. Symbolic Logic}, 53(3):696--707, 1988.

\bibitem{rodl}
V.~R\"{o}dl.
\newblock On the chromatic number of subgraphs of a given graph.
\newblock {\em Proc. Amer. Math. Soc.}, 64(2):370--371, 1977.

\bibitem{todorcevic_partitioning}
Stevo Todor\v{c}evi\'{c}.
\newblock Partitioning pairs of countable ordinals.
\newblock {\em Acta Math.}, 159(3-4):261--294, 1987.

\bibitem{todorcevic_combinatorial_cubes}
Stevo Todor\v{c}evi\'{c}.
\newblock Some partitions of three-dimensional combinatorial cubes.
\newblock {\em J. Combin. Theory Ser. A}, 68(2):410--437, 1994.

\end{thebibliography}
\end{document}